\documentclass[11pt]{article}
\usepackage{latexsym,amsmath,color,amsthm,amssymb,epsfig,graphicx,mathrsfs}
\usepackage{graphicx}
\usepackage{amssymb}
\usepackage[left=1in,top=1in,right=1in,bottom=1in]{geometry}
\usepackage[linktocpage=true]{hyperref}
\usepackage{setspace}
\usepackage{amssymb, amsmath, amsthm, graphicx,mathrsfs}
\usepackage{caption}

\usepackage{comment}
\def\qed{\hfill\ifhmode\unskip\nobreak\fi\quad\ifmmode\Box\else\hfill$\Box$\fi}
\def\ite#1{\hfill\break${}$\hbox to 50pt {\quad(#1)\hfill}}
\def\cA{{\mathcal A}}
\def\cB{{\mathcal B}}

\def\cF{{\mathcal F}}

\def\cH{{\mathcal H}}

\def\SDRP{{\rm SDRP }}
\newtheorem{thm}{Theorem}[section]
\newtheorem{cor}[thm]{Corollary}

\newtheorem{definition}[thm]{Definition}

\newtheorem{lem}[thm]{Lemma}
\newtheorem{conj}[thm]{Conjecture}

\def\ex{{\rm{ex}}}

\begin{document}

\pagestyle{myheadings} 
\markright{{\small{\sc Z.~F\"uredi, A.~Kostochka, and Ruth Luo:   Avoiding long Berge cycles}}}

\title{\vspace{-0.5in} Avoiding long Berge cycles}

\author{
{{Zolt\'an F\" uredi}}\thanks{
\footnotesize {Alfr\' ed R\' enyi Institute of Mathematics, Hungary.
E-mail:  \texttt{z-furedi@illinois.edu}. 
Research supported in part by the Hungarian National Research, Development and Innovation Office NKFIH grant K116769, and  by the Simons Foundation Collaboration Grant 317487.
}}
\and
{{Alexandr Kostochka}}\thanks{
\footnotesize {University of Illinois at Urbana--Champaign, Urbana, IL 61801
 and Sobolev Institute of Mathematics, Novosibirsk 630090, Russia. E-mail: \texttt {kostochk@math.uiuc.edu}.
 Research 
is supported in part by NSF grant  DMS-1600592
and grants 18-01-00353A and 16-01-00499  of the Russian Foundation for Basic Research.
}}
\and{{Ruth Luo}}\thanks{University of Illinois at Urbana--Champaign, Urbana, IL 61801, USA. E-mail: {\tt ruthluo2@illinois.edu}.}}

\date{\today}

\maketitle

\vspace{-0.3in}

\begin{abstract}
Let $n\geq k\geq r+3$ and $\cH$ be an $n$-vertex $r$-uniform hypergraph.
We show that if
   \[|\cH|> \frac{n-1}{k-2}\binom{k-1}{r}\]
then $\cH$ contains a Berge cycle of length at least $k$.
This bound is tight when $k-2$ divides $n-1$. We also show that the bound is attained only 
for connected $r$-uniform hypergraphs in which every  block is the complete hypergraph $K^{(r)}_{k-1}$.

We conjecture that our bound also holds in the case $k=r+2$, but the case of short cycles, $k\leq r+1$,
  is different.

\end{abstract}

\section{ \bf  Definitions, Berge $F$ subhypergraphs}

An $r$-uniform hypergraph, or simply {\em $r$-graph}, is a family of $r$-element subsets of a finite set.
We associate an $r$-graph $\cH$ with its edge set and call its vertex set $V(\cH)$.
Usually we take $V(\cH)=[n]$, where $[n]$ is the set of first $n$ integers, $[n]:=\{ 1, 2, 3,\dots, n\}$.
We also use the notation $\cH\subseteq \binom{[n]}{r}$.

\begin{definition}[Anstee and Salazar~\cite{AS}, Gerbner and Palmer~\cite{GP}] For a graph $F$ with vertex set $\{v_1, \ldots, v_p\}$ and edge set $\{e_1, \ldots, e_q\}$, a hypergraph $\mathcal H$ contains a {\bf Berge $F$} if there exist distinct vertices $\{w_1, \ldots, w_p\} \subseteq V(\mathcal H)$ and  edges $\{f_1, \ldots, f_q\} \subseteq E(\mathcal H)$, such that if $e_i = v_{i_1} v_{i_2}$, then $\{w_{i_1}, w_{i_2}\} \subseteq f_i$. The vertices $\{w_1, \ldots, w_p\}$ are called the {\bf base vertices} of the Berge $F$.  
\end{definition}

Of particular interest to us are Berge cycles. 

\begin{definition} A {\bf Berge cycle} of length $\ell$ in a hypergraph is a set of $\ell$ distinct vertices $\{v_1, \ldots, v_\ell\}$ and $\ell$ distinct edges $\{e_1, \ldots, e_\ell\}$ such that $\{ v_{i}, v_{i+1} \}\subseteq   e_i$ with indices taken modulo $\ell$.

 A {\bf Berge path} of length $\ell$ in a hypergraph in a hypergraph is a set of $\ell+1$ vertices $\{v_1, \ldots, v_{\ell+1}\}$ and $\ell$ hyperedges $\{e_1, \ldots, e_{\ell}\}$ such that $\{ v_{i}, v_{i+1} \}\subseteq   e_i$ for all $1\leq i\leq \ell$.
\end{definition}

Let $\cH$ be a hypergraph  and  $p$ be an integer. The {\em $p$-shadow}, $\partial_p \cH$, is the collection of the $p$-sets that lie in some edge of $\cH$. In particular, we will often consider the 2-shadow $\partial_2 \cH$ of a $r$-uniform hypergraph $\cH$ in which each edge of $\cH$ yields a clique on $r$ vertices. 


\section{Background}

Erd\H{o}s and Gallai~\cite{ErdGal59} proved the following result on the Tur\'an number of paths.

\begin{thm}[Erd\H{o}s and Gallai~\cite{ErdGal59}]\label{teg1} Let $k\geq 2$ and let $G$ be an $n$-vertex graph with no path on $k$ vertices. Then $e(G) \leq (k-2)n/2$. 
\end{thm}
This theorem is implied by a stronger result for graphs with no long cycles.
\begin{thm}[Erd\H{o}s and Gallai~\cite{ErdGal59}]\label{EGpaths}Let $k\geq 3$ and let $G$ be an $n$-vertex graph with no cycle of length $k$ or longer. Then $e(G) \leq (k-1)(n-1)/2$.
\end{thm}

Gy\H{o}ri, Katona, and Lemons~\cite{GKL} extended Theorem~\ref{teg1} to Berge paths in $r$-graphs. The bounds depend on the relationship of $r$ and $k$.
\begin{thm}[Gy\H{o}ri, Katona, and Lemons~\cite{GKL}]\label{paths}
Suppose that $\mathcal H$ is an  $n$-vertex $r$-graph with no Berge path of length $k$. If $k \geq r+2 \geq 5$, then $e(\mathcal H) \leq \frac{n}{k}{k \choose r}$, and if $r \geq k \geq 3$, then $e(\mathcal H) \leq \frac{n(k-1)}{r+1}$. 
\end{thm}

Both bounds in Theorem~\ref{paths} are sharp for each $k$ and $r$ for infinitely many $n$.
The remaining case of $k = r + 1$ was settled later  by Davoodi, Gy\H ori, Methuku, and Tompkins~\cite{DGMT}: 
{\em if $\cH$ is an $n$-vertex $r$-graph with $|E(\cH)| > n$, then it contains a Berge path of length at least $r + 1$.}
Furthermore,  Gy\H ori,  Methuku,   Salia,  Tompkins and  Vizer~\cite{GMSTV}  have found a better upper bound on the
number of edges in  $n$-vertex {\em connected} $r$-graphs with no Berge path of length $k$. Their bound is  asymptotically exact when $r$ is fixed and
$k$ and $n$ are sufficiently large.

The goal of this paper is to present a similar result for cycles.

\section{\bf  Main result: Hypergraphs without long Berge cycles}

 Our main result is an analogue of the Erd\H{o}s--Gallai theorem on cycles for $r$-graphs. 



\begin{thm}\label{mainF}Let $r \geq 3$ and $k \geq r+3$, and suppose $\mathcal H$ is an  $n$-vertex $r$-graph with no Berge cycle of length $k$ or longer. Then $e(\mathcal H) \leq \frac{n-1}{k-2}{k-1 \choose r} $. Moreover, equality is achieved if and only if $\partial_2 \cH$ is connected and for every block $D$ of $\partial_2 \cH$, $D = K_{k-1}$ and $\mathcal H[D] = K_{k-1}^{(r)}$.
\end{thm}

\begin{center}
\includegraphics[scale=.4]{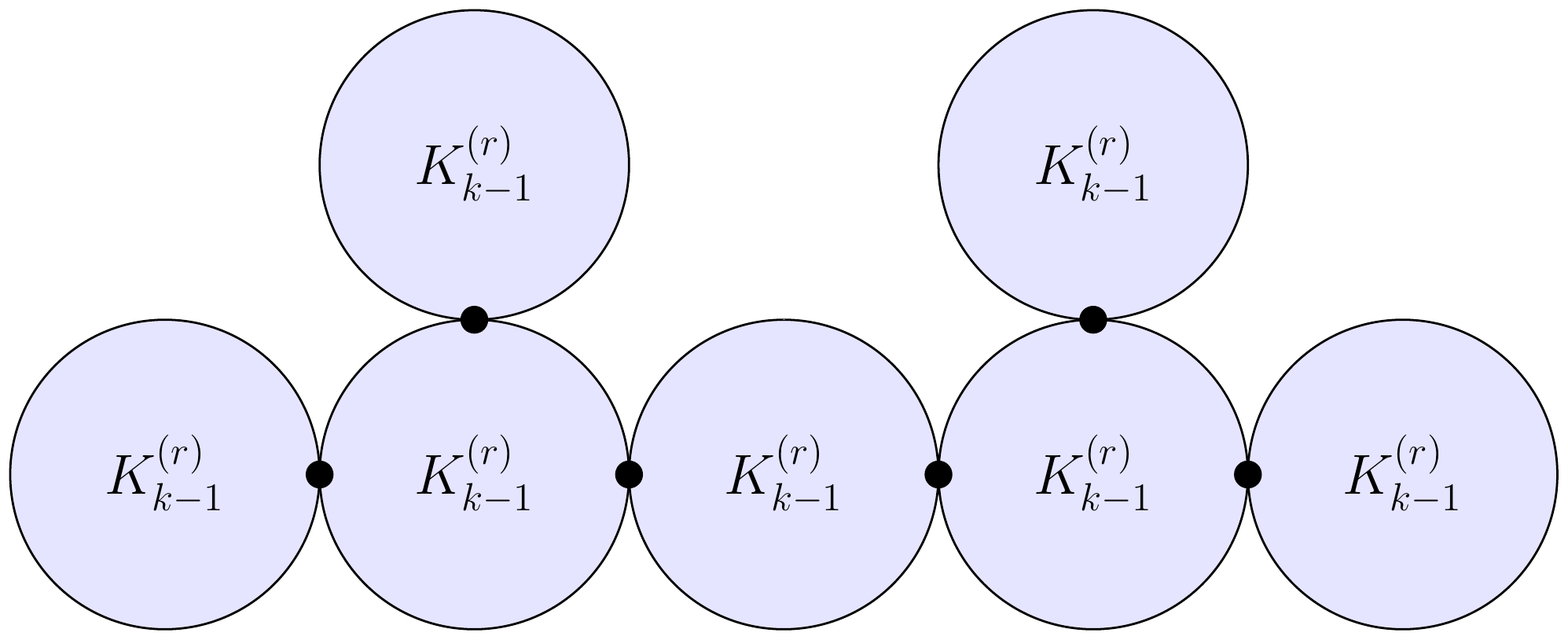}
\end{center}

Note that a Berge cycle can only be contained in the vertices of a single block of the 2-shadow. Hence the aforementioned sharpness examples cannot contain Berge cycles of length $k$ or longer. 


\begin{conj}\label{conjmain}
The statement of Theorem~\ref{mainF} holds for $k= r+2$, too.
  \end{conj}


Similarly to the situation with paths, the case of short cycles, $k\leq r+1$,  is different. Exact bounds for $k\leq r-1$ and asymptotic bounds for $k = r$ were found in~\cite{KosLuo}. The answer for $k=r+1$ is not known.

For convenience, below we will use notation
\begin{equation}\label{crk}
C_r(k): = \frac{1}{k-2}{k-1 \choose r}.
\end{equation}

(So $C_2(k)(n-1) = (k-1)(n-1)/2$.) Theorem~\ref{mainF} yields the following implication for paths.

\begin{cor}\label{corP} Let $r \geq 3$ and $n\geq k+1 \geq r+4$.
If $\mathcal H$ is a connected  $n$-vertex $r$-graph with no Berge path of length $k$, then $e(\mathcal H) \leq C_r(k){(n-1)}$. \end{cor}

This gives a $\frac{k-2}{k-r}$ times stronger  bound than Theorem~\ref{paths} for connected $r$-graphs for all $r \geq 3$ and $n\geq k+1 \geq r+4$
and not only for sufficiently large $k$ and $n$. In particular, Corollary~\ref{corP} implies the following slight sharpening of
Theorem~\ref{paths} for $k\geq r+3$.

\begin{cor}\label{corP2} Let $r \geq 3$ and $n\geq k \geq r+3$.
If $\mathcal H$ is an  $n$-vertex $r$-graph with no Berge path of length $k$, then $e(\mathcal H)\leq \frac{n}{k}{k \choose r}$
with equality only if every component of $\cH$ is the complete $r$-graph
$K^{(r)}_{k}$.
 \end{cor}


 In the next section, we introduce the notion of {\em representative pairs} and use it to derive useful properties of Berge $F$-free hypergraphs for rather general $F$. In Section~\ref{Kop}, we cite Kopylov's Theorem and prove two useful inequalities.
 In Section~\ref{mproof} we prove our main result, Theorem~\ref{mainF}, and in the final Section~\ref{corpa} we derive Corollaries~\ref{corP} and~\ref{corP2}.

\section{\bf  Representative pairs, the structure of Berge $F$-free hypergraphs}

\begin{definition}
For a hypergraph $\mathcal H$, a {\bf system of distinct representative pairs (SDRP) of $\mathcal H$} is a set of distinct pairs $A = \{\{x_1, y_1\}, \ldots, \{x_s, y_s\}\}$ and a set of distinct hyperedges 
$\cA = \{f_1, \ldots f_s\}$ of $\mathcal H$ such that for all $1\leq i \leq s$

${}$\quad ---  $\{x_i, y_i\} \subseteq f_i$, and\\
${}$\quad ---  $\{x_i, y_i\}$ is not contained in any $f \in \mathcal H - \{f_1, \ldots, f_s\}$.


\end{definition}

 \begin{lem}\label{le:maxA} Let $\mathcal H$ be a 
 hypergraph, let $(A, \cA)$ 
be an \SDRP of $\mathcal H$ of maximum size. 
Let $\cB:= \cH \setminus \cA$ and
let ${B} = \partial_2 \cB$ be the $2$-shadow of ${\mathcal B}$.
For a subset $S \subseteq B$, let ${\mathcal{B}_S}$ denote the set of hyperedges that contain at least one edge of $S$. Then for all nonempty $S \subseteq B$, $|S|< |{\mathcal B}_S|$.
 \end{lem}

 \begin{proof}Suppose for contradiction there exists a nonempty set $S \subseteq B$ such that $|S| \geq  |{\mathcal B}_S|$. Choose a smallest such $S$.

We claim that $|S|=|\cB_S|$. Indeed, if $|S| > |{\mathcal B}_S|$  then $|S|\geq 2$ because $\cB_S\neq \emptyset$ by definition.
Take any edge $e \in S$. The set $S \setminus e$ is nonempty and $|S\setminus e| = |S|-1\geq |{\mathcal B}_S|\geq  |{\mathcal B}_{S\setminus e}|$, a contradiction to the minimality of $S$.

Consider the case $|S| = |{\mathcal B}_S|$.
By the minimality of $S$, each subset $S' \subset S$ satisfies $|S'|< |\cB_{S'}|$.
Therefore by Hall's theorem,  one can find a bijective mapping of $S$ to ${\mathcal B}_S$, where say the edge $e_i \in S$ gets mapped to hyperedge $f_{i}$ in ${\mathcal B}_S$ for $1\leq j \leq |S|$.
Then $(A \cup \{e_{i}, \ldots, e_{|S|}\}, \mathcal A \cup \{ f_{1}, \ldots, f_{{|S|}}\})$ is a larger $\SDRP$ of $\mathcal H$, a contradiction.
 \end{proof}

\begin{lem}\label{le:BergeFinG}
Let $\mathcal H$ be a 
 hypergraph and let $(A, \cA)$ 
be an \SDRP of $\mathcal H$ of maximum size.
Let $\cB:= \cH \setminus \cA$,
 ${B} = \partial_2 \cB$, and let $G$ be the graph on $V(\cH)$ with edge set $A\cup B$.
If $G$  contains a copy of a graph $F$, then $\mathcal H$ contains a Berge $F$ on the same base vertex set.
\end{lem}

\begin{proof}
Let $\{v_1, \ldots, v_p\}$ and $\{e_1, \ldots, e_q\}$ be a set of vertices and a set of edges forming a copy of $F$ in $G$ such that the edges $e_1, \dots , e_b$ belong to $B$.
By Lemma~\ref{le:maxA}, each subset $S$ of $\{e_1, \ldots, e_b\}$ satisfies $|S| < |{\mathcal B}_S|$. So we may apply Hall's Theorem to match each of these $e_i$'s to a hyperedge $f_i \in {\mathcal B}$.
The edges $e_i\in A$ can be matched to distinct edges of $\cA$ given by the SDRP.
Since $\cA\cap \cB=\emptyset$ this yields a Berge $F$ in ${\mathcal H}$ on the same base vertex set.
\end{proof}

We have $|\mathcal H| = |A| + |{\mathcal B}|$. Note that the number of $r$-edges in $\cB$ is at most the number of copies of $K_r$ in its 2-shadow.  Therefore Lemma~\ref{le:BergeFinG} gives a new proof for the following result of Gerbner and Palmer (cited in~\cite{GMV}):
for any graph $F$,
\[\ex(n,K_r, F) \leq \ex_r(n, \text{Berge } F) \leq \ex(n, F)+ \ex(n,K_r, F) .
\]

Here $\ex_r(n,\{ \cF_1, \cF_2, \dots \})$ denotes the {\em Tur\'an number of $\{ \cF_1, \cF_2, \dots \}$}, the maximum number of edges in an $r$-uniform hypergraph on $n$ vertices that does not contain a copy of any $\cF_i$.

The  {\em generalized Tur\'{a}n function} $\ex(n,K_r,F)$ is the maximum number of  copies of $K_r$ in an $F$-free graph on $n$ vertices.

\section{\bf  Kopylov's Theorem and two inequalities}\label{Kop}

{\bf Definition}: For a natural number $\alpha$ and a graph $G$, the \emph{$\alpha$-disintegration} of
a graph $G$ is the  process of iteratively removing from $G$ the vertices with degree
at most $\alpha$  until the resulting graph has minimum degree at least $\alpha + 1$ or is empty.
This resulting subgraph $H(G, \alpha)$ will be called the $(\alpha+1)$-{\em core} of $G$. It is well known (and easy)
that $H(G, \alpha)$ is unique and does not depend on the order of vertex deletion.


The following theorem is a consequence of Kopylov~\cite{Kopy} about the structure of graphs without long cycles. We state it in the form that we need.\footnote{A proof and a recent application can be found in~\cite{luo}.}

\begin{thm}[Kopylov~\cite{Kopy}]\label{le:Kopy}
 Let $n \geq k \geq 5$ and let $t = \lfloor \frac{k-1}{2}\rfloor$.
Suppose that $G$ is a $2$-connected $n$-vertex graph with no  cycle of length at least $k$.
Suppose that it is saturated, i.e., for every nonedge $xy$ the graph $G \cup \{ xy\}$ has a cycle
 of length at least $k$. Then either \\
${}$\quad {\rm (\ref{le:Kopy}.1)} \enskip the  $t$-core $H(G, t)$ is empty, the graph $G$ is $t$-disintegrable; or \\
${}$\quad {\rm (\ref{le:Kopy}.2)} \enskip $|H(G, t)|=s$ for some
$t+2\leq s\leq k-2$, it is a complete graph on $s$ vertices, and $H(G,t)= H(G, k-s)$, i.e.,
the rest of the vertices can be removed by a $(k-s)$-disintegration.
\end{thm}

Note that in the second case $2\leq k-s \leq t$.

\begin{lem}\label{lem:upbnd}
Let $k,r,t,s,a$ nonnegative integers, and suppose $k\geq r+3\geq 6$, $t=\lfloor (k-1)/2\rfloor$, and
  $0\leq a \leq s\leq t$. Then
\[   a+ \binom{s-a}{r-1} \leq \frac{1}{k-2}\binom{k-1}{r}:= C_r(k).
\]
   \end{lem}
This is the part of the proof where we use $k\geq r+3$ because this inequality does not hold for $k=r+2$ (then the right hand side is $(r+1)/r$ while the left hand side could be as large as $\lfloor (r+1)/2\rfloor$).

\begin{proof}
Keeping $k,r,t,s$ fixed the left hand side is a convex function of $a$
(defined on the integers $0\leq a \leq s$).
It takes its maximum either at $a=s$ or $a=0$.
So the left hand side is at most $\max\{  s, \binom{s}{r-1}\}$.
This is at most $\max\{ t, \binom{t}{r-1}\}$.
We have eliminated the variables $a$ and $s$.

We claim that $t\leq  \frac{1}{k-2}\binom{k-1}{r}$.
Indeed, keeping $k,t$ fixed, the right hand side is minimized when $r=k-3$, and then it equals to
  $(k-1)/2$. This is at least $\lfloor (k-1)/2\rfloor =t$.

Finally, we claim that
  $\binom{t}{r-1} \leq \frac{1}{k-2}\binom{k-1}{r}$. If $t< r-1$, then there is nothing to prove.
For $t\geq r-1$ rearranging the inequality we get
\[    r \leq \frac{k-1}{t}\times \frac{k-3}{t-1}\times \dots \times \frac{k-r}{t-r+2}.
\]
Each fraction on the right hand side  is at least $2$. Since  $r<2^{r-1}$,  we are done.
 \end{proof}

\begin{lem}\label{lem:shadow}
Let $w, \, r\geq 2$ and let $\cH$ be a $w$-vertex $r$-graph.
Let $\overline{\partial_2 \cH}$ denote the family of pairs of $V(\cH)$ not contained in any member of $\cH$ (i.e., the complement of the $2$-shadow).
Then
\[
     |\cH| +    | \overline{\partial_2 \cH}| \leq
  a_r(w):=  
\left\{\begin{array}{ll}\dbinom{w}{2}        &\mbox{for  $2\leq w\leq r+2$},\\
   \dbinom{w}{r}   &\mbox{for $r+2\leq w $}. \end{array} \right.
\]
Moreover, for $2\leq w\leq k-1$ one has $a_r(w)\leq (w-1) \binom{k-1}{r}/(k-2)$ with equality if and only if $w = k-1$ and 

${}$\quad ---  $w > r+2$ and $\mathcal H$ is complete, or\\
${}$\quad ---  $w = r+2$ and either $\mathcal H$ or $\overline{\partial_2 \cH}$ is complete.
  \end{lem}
\begin{proof}
The case of $w\geq r+2$  is a corollary of the classical Kruskal-Katona theorem, but one can give a direct proof by a double counting. 
If $\overline{\partial_2 \mathcal H}$ is empty, then $|\mathcal H| = {w \choose r}$ if and only if $\mathcal H = {V(\cH) \choose r}$. Otherwise, let $\overline{\cH}$ denote the $r$-subsets of $V(\cH)$ that are not members of $\cH$,
   $\overline{\cH}= \binom{V(\cH)}{r}\setminus \cH$.
Each pair of $\overline{\partial_2 \cH}$ is contained in $\binom{w-2}{r-2}$ members of $\overline{\cH}$ and each $e\in \overline{\cH}$ contains at most $\binom{r}{2}$ edges of $\overline{\partial_2 \cH}$. We obtain
\[   |\overline{\partial_2 \cH}|\binom{w-2}{r-2}\leq |\overline{\cH}|\binom{r}{2}.
\]
Since $\binom{w-2}{r-2}\geq \binom{r}{r-2}= \binom{r}{2}$,  $|\overline{\partial_2 \cH}|\leq |\overline{\cH}|$ with equality only when $w = r+2$. Furthermore, if $\overline{\partial_2\cH}$ and $\cH$ are both nonempty, then for any $xy \in \overline{\partial_2\cH}$ and $uv \in \partial_2\cH$ (with possibly $x = u$), any $r$-tuple $e$ containing $\{x,y\} \cup \{u,v\}$ is in $\overline{\cH}$ but contributes strictly less than ${r \choose 2}$ edges to $\overline{\partial_2 \cH}$, implying $| \overline{\partial_2 \cH}| < |\overline{\cH}|$. This completes the proof of the case.

The case $w\leq r+1$ is easy, and the calculation showing $a_r(w)\leq C_r(k) (w-1)$ with equality only if $w = k-1$ is standard.
   \end{proof}

\section{\bf  Proof of Theorem~\ref{mainF}, the main upper bound} \label{mproof}

\begin{proof}
Let $\mathcal H$ be an $r$-uniform hypergraph on $n$ vertices with no Berge cycle of length $k$ or longer
 ($k \geq r+3\geq 6$).
Let $(A, \cA)$
be an \SDRP of $\mathcal H$ of maximum size.
Let $\cB:= \cH \setminus \cA$,
 ${B} = \partial_2 \cB$.
By Lemma~\ref{le:BergeFinG} the graph $G$ with edge set $A\cup B$ does not contain a cycle of length $k$ or longer.

Let $V_1, V_2, \dots, V_p$ be the vertex sets of the standard (and unique) decomposition of $G$ into $2$-connected blocks of sizes $n_1, n_2, \dots, n_p$.
Then the graph $A\cup B$ restricted to $V_i$, denoted by $G_i$, is either a $2$-connected graph or a single edge (in the latter case $n_i=2$),
 each edge from $A\cup B$ is contained in a single $G_i$, and $\sum_{i=1}^p (n_i-1)\leq (n-1)$.

This decomposition yields a decomposition of $A=A_1\cup A_2\cup \dots\cup A_p$ and
  $B=B_1\cup B_2\cup \dots\cup B_p$, $A_i\cup B_i=E(G_i)$.
If an edge $e\in B_i$ is contained in $f\in \cB$, then $f\subseteq V_i$ (because $f$ induces a
$2$-connected graph $K_r$ in $B$), so the block-decomposition of $G$ naturally extends to $\cB$,
  $\cB_i:= \{ f\in \cB: f\subseteq V_i\}$ and we have $\cB= \cB_1\cup \dots \cup \cB_p$, and $B_i=\partial_2 \cB_i$.

We claim that for each $i$,
\begin{equation}\label{eq:5}
|A_i|+|\cB_i|\leq C_r(k)(n_i-1),
\end{equation}
and hence \[|\mathcal H| = |A| + |\mathcal B| = \sum_{i=1}^p |A_i| + |\cB_i| \leq \sum_{i=1}^p C_r(k)(n_i-1) \leq C_r(k)(n-1),\] completing the proof. 

To prove~\eqref{eq:5} observe that the case $n_i\leq k-1$ immediately follows from Lemma~\ref{lem:shadow}.
From now on, suppose that $n_i\geq k$.

Consider the graph $G_i$ and, if necessary, add edges to it to make it a saturated graph with no cycle of length $k$ or longer. Let the resulting graph be $G'$. 
Kopylov's Theorem (Theorem~\ref{le:Kopy}) can be applied to $G'$.
If $G$ is $t$-disintegrable, then make $(n_i-k+2)$ disintegration steps and let $W$ be the remaining vertices of $V_i$ ($|W|=k-2$).
For the edges of $A_i$ and $\cB_i$ contained in $W$ we use Lemma~\ref{lem:shadow} to see that
  \[    |A_i[W]|+ |\cB_i[W]|< C_r(k)(|W|-1).
\]
In the $t$-disintegration steps, we iteratively remove vertices with degree at most $t$ until we arrive to $W$. When we remove a vertex $v$ with degree $s \leq t$ from $G'$, $a$ of its incident edges are from $A$, and the remaining $s-a$ incident edges eliminate at most $\binom{s-a}{r-1}$ hyperedges from $\cB_i$ containing $v$. Therefore $v$ contributes at most $a + {s-a \choose r-1} \leq C_r(k)$ (by Lemma~\ref{lem:upbnd}) to $|\cB_i| + |A_i|$. 

It follows that \[|A_i| + |\mathcal B_i| < \left(\sum_{v \in G' - W} C_r(k)\right) + C_r(k)(|W|-1)  = C_r(k)(n_i - 1).\] This completes this case.

Next consider the case  (\ref{le:Kopy}.2), $W:=V(H(G,t))$, $|W|=s\leq k-2$. We proceed as in the previous case,
  making $(n_i-s)$ disintegration steps. Apply  Lemma~\ref{lem:shadow} for  $|A_i[W]|+ |\cB_i[W]|$
 and Lemma~\ref{lem:upbnd} for the $(k-s)$-disintegration steps (where $k-s \leq t$) to get the desired upper bound (with strict inequality).
 
Furthermore, if $e(\mathcal H)  = |A| + |\mathcal B| = C_r(k)(n-1)$, then we have $\sum_{i=1}^p (n_i-1) = n-1$ (so $A \cup B$ is connected) and $|A_i| + |\mathcal B_i| = C_r(k) (n_i-1)$ for each $1 \leq i \leq p$. From the previous proof and Lemma~\ref{lem:upbnd}, we see that this holds if and only if for each $i$, $n_i = k-1$, and either $\mathcal B_i$ or $A_i$ is complete. In particular, this implies that each block of $A \cup B$ is a $K_{k-1}$. We will show that  each $G_i$ corresponds to a block in
in $\mathcal H$ that is $K_{k-1}^{(r)}$ with vertex set $V_i$. 

In the case that $\mathcal B_i$ is complete for all $1 \leq i \leq p$, we are done.
Otherwise, if some $A_i$ is complete (note $r=k-3$ by Lemma~\ref{lem:upbnd}) then there are ${k-1 \choose 2} = {k-1 \choose k-3} = {k-1 \choose r}$ hyperedges in $\mathcal A$ containing $V_i$. If all such hyperedges are contained in $V_i$, again we get $\mathcal H[V_i] = K_{k-1}^{(r)}$. So suppose there exists a $f \in \mathcal A$ which is paired with an edge $xy \in A_i$ in the SDRP, but for some $z \notin V_i$, $\{x,y,z\} \subseteq f$. Then $z$ belongs to another block $G_j$ of $A \cup B$. In $A \cup B$, there exists a path from $x$ to $z$ covering $V_i \cup V_j$ which avoids the edge $xy$. Thus by Lemma~\ref{lem:shadow}, there is a Berge path from $x$ to $z$ with at least $2(k-1) - 1$ base vertices which avoids the hyperedge $f$ (since edge $xy$ was avoided). Adding $f$ to this path yields a Berge cycle of length $2(k-1)-1 > k$, a contradiction. 
\end{proof}

\section{Corollaries for paths}\label{corpa}

In order to be self-contained, we present a short proof of a lemma by Gy\H{o}ri, Katona, and Lemons~\cite{GKL}.
 
\begin{lem}[Gy\H{o}ri, Katona, and Lemons~\cite{GKL}]\label{GKL}
Let $\mathcal H$ be a connected hypergraph with no Berge path of length $k$. If there is a Berge cycle of length $k$ on the vertices $v_1, \ldots, v_{k}$ then these vertices constitute a component of $\mathcal H$. 
\end{lem}

\begin{proof}

Let $V=\{v_1,\ldots,v_k\}$, $E=\{e_1,\ldots,e_k\}$ form the Berge cycle in $\cH$. If some edge, say $e_1$ contains a vertex $v_0$ outside of $V$, then we have a path with vertex set 
$\{v_0,v_1,\ldots,v_\ell\}$ and edge set $E$. Therefore each $e_i$ is contained in $V$. Suppose $V \neq V(\mathcal H)$. Since $\cH$ is connected,
there exists an edge $e_0\in\cH$ and a vertex $v_{k+1}\notin V$ such that for some $v_i\in V$, say $i=k$, 
$\{v_k,v_{k+1}\}\subseteq e_0$. Then $\{v_1, \ldots, v_k, v_{k+1}\}, \{e_1, \ldots, e_{k-1}, e_0\}$ is a Berge path of length $k$.
\end{proof}

{\em Proof of Corrollary~\ref{corP}.} Suppose $n\geq k+1$ and $\mathcal H$ is a connected  $n$-vertex $r$-graph with  $e(\mathcal H) >C_r(k){(n-1)}$.
Then by Theorem~\ref{mainF}, $\mathcal H$ has a Berge cycle of length $\ell\geq k$.
If  $\ell\geq k+1$, then removing any edge from the cycle yields a Berge path of length at least $k$. If $\ell=k$, then by Lemma~\ref{GKL}, $\mathcal H$ again has
 a Berge path of length $k$.\qed

\medskip
Now Theorem~\ref{mainF} together with Corollary~\ref{corP} directly imply Corollary~\ref{corP2}.

\medskip
{\bf Proof of Corollary~\ref{corP2}:}
Suppose $k \geq r+3 \geq 6$ and
 $\mathcal H$ is
 an $r$-graph. 
Let $\cH_1, \cH_2, \ldots,\cH_s$ be the connected components of $\cH$ and $|V(\cH_i)|=n_i$ for $i=1,\ldots,s$.

If $n_i\leq k-1$, then $|\cH_i|\leq {n_i \choose r} < \frac{n_i}{k}{k \choose r}$. 
If $n_i\geq k+1$, then by  Corollary~\ref{corP},   $|\cH_i| \leq C_r(k)(n_i-1) < \frac{n_i}{k}{k \choose r}$.
Finally, if $n_i= k$, then $|\cH_i|\leq {k \choose r} = \frac{n_i}{k}{k \choose r}$, with equality only if $\cH_i=K^{(r)}_k$.
This proves the corollary.\qed

\vspace{12mm}
{\bf Acknowledgment.} The authors would like to thank Jacques Verstra\"ete for suggesting this problem and for sharing his ideas and methods used in similar problems.

\end{document}